\documentclass[11pt]{article}
\usepackage[T1]{fontenc}
\usepackage[utf8]{inputenc}
\usepackage{lmodern}
\usepackage{amsmath,amssymb,amsthm}
\usepackage{booktabs}
\usepackage{geometry}
\usepackage{hyperref}
\usepackage{microtype}
\usepackage{xcolor}
\hypersetup{colorlinks=true,linkcolor=blue!55!black,citecolor=blue!55!black,urlcolor=blue!55!black}

\newtheorem{theorem}{Theorem}
\newtheorem{corollary}[theorem]{Corollary}
\newtheorem{lemma}[theorem]{Lemma}

\newcommand{\F}{\mathcal F}
\newcommand{\E}{\mathcal E}
\newcommand{\Oset}{\mathcal O}

\title{A 3-semi-perfect 1-factorization of the six-dimensional hypercube}
\author{Guillaume Lambard\\{}Data-driven Materials Design Group,\\{}Center for Basic Research on Materials (CBRM),\\{}National Institute for Materials Science (NIMS),\\{}Ibaraki, Tsukuba, Namiki 1-1, 305-0044 Japan\\\texttt{LAMBARD.Guillaume@nims.go.jp}}
\date{16 July 2026}

\begin{document}
\maketitle

\begin{abstract}
For a 1-factorization $\mathcal F=\{M_1,\ldots,M_d\}$ of the hypercube $Q_d$, let
$G[\mathcal F]$ have vertex set $\mathcal F$, with $M_iM_j$ an edge exactly when
$M_i\cup M_j$ is a Hamilton cycle. Behague proved that $Q_{k+\ell}$ has a
1-factorization $\mathcal F$ with $G[\mathcal F]\cong K_{k,\ell}$ for all positive
$k,\ell$ except possibly $k=\ell=3$. We give an explicit 1-factorization of $Q_6$
for which $G[\mathcal F]\cong K_{3,3}$, resolving the exceptional case. The
construction is supplied as a finite certificate. Its correctness can be checked
directly from the tables in this paper or by either of two independent, short
standard-library verifiers supplied with the certificate.
\end{abstract}

\noindent\textbf{Keywords:} factorization; hypercube; Hamilton cycle; semi-perfect
1-factorization; computer-assisted proof.

\noindent\textbf{2020 Mathematics Subject Classification:} 05C70, 05C45.

\section{Introduction}

Let $Q_d$ be the graph with vertex set $\{0,1\}^d$, in which two vertices are
adjacent when they differ in exactly one coordinate. A \emph{1-factorization} of a
$d$-regular graph is a partition of its edge set into $d$ perfect matchings. Given a
1-factorization $\F=\{M_1,\ldots,M_d\}$ of $Q_d$, define its Hamiltonian-pair graph
$G[\F]$ on vertex set $\F$ by
\[
 M_iM_j\in E(G[\F])
 \quad\Longleftrightarrow\quad
 M_i\cup M_j\text{ is a Hamilton cycle of }Q_d.
\]
The factorization is $k$-semi-perfect when $G[\F]$ contains a copy of
$K_{k,d-k}$ with the prescribed two parts.

Behague proved that $G[\F]$ is bipartite for every 1-factorization of a hypercube
and constructed a factorization satisfying
$G[\F]\cong K_{k,\ell}$ in $Q_{k+\ell}$ for every pair of positive integers
$(k,\ell)\ne(3,3)$ \cite{behague2019}. The case $G[\F]\cong K_{3,3}$ in $Q_6$
was left as the sole exception. The construction below closes that exception.

\begin{theorem}\label{thm:main}
The six-dimensional hypercube has a 1-factorization $\F$ such that
$G[\F]\cong K_{3,3}$. In particular, $Q_6$ has a 3-semi-perfect
1-factorization.
\end{theorem}

Combining Theorem~\ref{thm:main} with Behague's theorem gives the complete
existence statement.

\begin{corollary}
For every pair of positive integers $k$ and $\ell$, the hypercube $Q_{k+\ell}$
has a 1-factorization $\F$ for which $G[\F]\cong K_{k,\ell}$.
\end{corollary}

\section{A checkable certificate}

We identify a vertex $(x_0,\ldots,x_5)\in\{0,1\}^6$ with the integer
$\sum_{r=0}^5x_r2^r\in\{0,\ldots,63\}$. Thus $u$ and $v$ are adjacent exactly
when $u\mathbin{\mathtt{xor}}v\in\{1,2,4,8,16,32\}$. Appendix~\ref{app:witness}
lists six sets $M_1,\ldots,M_6$ of 32 edges. Endpoints are written as two-digit
hexadecimal numbers; for example, \texttt{0222} denotes the edge
$\{\mathtt{02},\mathtt{22}\}=\{2,34\}$.

Let $\E$ and $\Oset$ be the even- and odd-parity vertices of $Q_6$. A perfect
matching $M_i$ determines a bijection $m_i:\E\to\Oset$. For distinct $i,j$, put
\[
 \pi_{j,i}=m_j^{-1}m_i:\E\longrightarrow\E.
\]

\begin{lemma}\label{lem:perm}
The union $M_i\cup M_j$ is a Hamilton cycle of $Q_6$ if and only if
$\pi_{j,i}$ is a 32-cycle on $\E$.
\end{lemma}

\begin{proof}
Starting at $u\in\E$ and following first the $M_i$-edge and then the $M_j$-edge
returns to $\E$ at $m_j^{-1}m_i(u)=\pi_{j,i}(u)$. Hence cycles of length $r$ in
$\pi_{j,i}$ correspond exactly to alternating cycles of length $2r$ in
$M_i\cup M_j$. The union is one cycle on all 64 vertices exactly when
$\pi_{j,i}$ is one cycle on all 32 even vertices.
\end{proof}

\begin{proof}[Proof of Theorem~\ref{thm:main}]
Use the six edge sets in Appendix~\ref{app:witness}. Each row contains 32 cube
edges and each vertex occurs exactly once in each row, so each row is a perfect
matching. The 192 displayed edges are distinct. Since
$|E(Q_6)|=6\cdot64/2=192$, the six matchings partition $E(Q_6)$ and form a
1-factorization.

Appendix~\ref{app:cycles} displays $\pi_{j,i}$ as a single cycle containing all
32 even vertices for every $i\in\{1,2,3\}$ and $j\in\{4,5,6\}$. Lemma
\ref{lem:perm} therefore shows that all nine cross-pairs are Hamiltonian.
For each of the six pairs within $\{1,2,3\}$ or within $\{4,5,6\}$,
Appendix~\ref{app:nonedges} exhibits a proper alternating cycle in the union.
None of those six unions is Hamiltonian. Consequently the Hamiltonian-pair
graph has precisely the nine cross-edges, and hence is $K_{3,3}$.
\end{proof}

\section{Independent verification and provenance}

The appendices constitute a self-contained finite proof certificate. For
convenience and error detection, the supplementary archive also contains the
same factorization in JSON and two independently written Python verifiers. Each
verifier checks directly that every listed pair is a cube edge, that every row is
a perfect matching, that the six rows partition all 192 edges, and that every
cross-pair union is a 64-cycle. Neither verifier searches for a factorization,
and neither imports the other. Both use only the Python standard library.

With the supplementary files in the working directory, the checks are
reproduced by
\begin{verbatim}
python3 verifier.py candidate.json
python3 verifier_cleanroom.py candidate.json
sha256sum -c SHA256SUMS
\end{verbatim}
The SHA-256 digest of the JSON certificate is
\begin{center}
\small\ttfamily
3d3a476aa481ac0e5a0fd51e3f4a677d8dadf5ec3c000ed0f972dee1982bf8db.
\end{center}

\paragraph{Discovery of the certificate.}
The certificate was obtained in two computational stages. First, a Boolean
constraint search represented a 1-factorization as a proper six-edge-colouring
of $Q_6$, imposed a sound orbit normalization and the necessary
factorization-sign condition, and added necessary connectivity cuts whenever a
required bichromatic union was disconnected. This produced a near-witness in
which eight of the nine required cross-pairs were Hamiltonian and the remaining
union had cycle lengths $8$, $12$, and $44$.

Starting from this near-witness, a seeded beam search applied Kempe-component
trades. For two colours $a$ and $b$, interchanging them on any connected
component of the subgraph formed by their colour classes preserves the
1-factorization. Allowing trades between all fifteen colour pairs, the frozen
replay found the displayed certificate at trade depth eight after visiting
$599{,}257$ distinct factorizations. A separate end-to-end rerun generated a
different eight-of-nine near-witness after $5{,}067$ constraint-solver rounds
and $89{,}540$ connectivity cuts, and repaired it to a second independently
accepted certificate at trade depth twelve. The programs, inputs, complete
logs, parameters, hashes, and both outputs are preserved in the accompanying
repository.

These computations were used only to discover the displayed object. They are
not logical dependencies of the proof: only the explicit certificate and the
checks above are used to establish Theorem~\ref{thm:main}.

The exact certificate, its discovery, and verifier versions used for this result are preserved in the immutable Zenodo archive \cite{lambard2026q6archive}.

\section*{Data and code availability}

The complete frozen reproducibility package, including the explicit factorization certificate, two independent verifiers, SHA-256 hashes, and
reproduction instructions, is permanently archived on Zenodo at
\href{https://doi.org/10.5281/zenodo.21392425} {doi:10.5281/zenodo.21392425} \cite{lambard2026q6archive}. The associated public development repository is available at \url{https://github.com/GLambard/q6-semi-perfect-factorization}.

\section*{Funding}

No specific funding was received for performing the present work.

\section*{Declaration of competing interests}

There are no competing interests to be declared. 

\section*{Declaration of generative AI and AI-assisted technologies in the
manuscript preparation process}

During the preparation of this work, the author used OpenAI Terra/Sol 5.6 (High) models to assist computational exploration, independent-code drafting, literature organization, and manuscript preparation. The author reviewed and verified the mathematical certificate, software outputs, references, and manuscript, and take full responsibility for the content of the work.

\appendix
\section{The six matchings}\label{app:witness}

In the following table each four-character hexadecimal word \texttt{uuvv}
denotes the unordered edge $\{\mathtt{uu},\mathtt{vv}\}$.

\begin{center}
\begin{minipage}{0.98\textwidth}
\small\ttfamily\raggedright
$M_1$: 0001, 0222, 0307, 0406, 0525, 0828, 090B, 0A2A,
0C0E, 0D1D, 0F2F, 1018, 1113, 1216, 1415, 1737,
1939, 1A3A, 1B1F, 1C1E, 2030, 2131, 2327, 242C,
262E, 292D, 2B3B, 3233, 3435, 363E, 383C, 3D3F.

\medskip
$M_2$: 0020, 0121, 0206, 030B, 040C, 0515, 0717, 0818,
0929, 0A1A, 0D0F, 0E1E, 1012, 1119, 131B, 1416,
1C3C, 1D1F, 2223, 2426, 2535, 272F, 2838, 2A2B,
2C2E, 2D3D, 3032, 3133, 3436, 373F, 393B, 3A3E.

\medskip
$M_3$: 0004, 0103, 020A, 050D, 0616, 0727, 080C, 0919,
0B1B, 0E0F, 1011, 1213, 1434, 1517, 181C, 1A1E,
1D3D, 1F3F, 2024, 2129, 2226, 2333, 252D, 282A,
2B2F, 2C3C, 2E3E, 3038, 3139, 3236, 3537, 3A3B.

\medskip
$M_4$: 0008, 0109, 0203, 0414, 0507, 060E, 0A0B, 0C0D,
0F1F, 1030, 1115, 121A, 1333, 1617, 1819, 1B3B,
1C1D, 1E3E, 2021, 2232, 232B, 2425, 2627, 2829,
2A2E, 2C2D, 2F3F, 3135, 343C, 3637, 383A, 393D.

\medskip
$M_5$: 0002, 0111, 0323, 0405, 0626, 070F, 080A, 090D,
0B2B, 0C1C, 0E2E, 1014, 1232, 1317, 1535, 1636,
1838, 191D, 1A1B, 1E1F, 2022, 2125, 2434, 2737,
282C, 2939, 2A3A, 2D2F, 3031, 333B, 3C3D, 3E3F.

\medskip
$M_6$: 0010, 0105, 0212, 0313, 0424, 0607, 0809, 0A0E,
0B0F, 0C2C, 0D2D, 1131, 141C, 151D, 161E, 171F,
181A, 191B, 2028, 2123, 222A, 2527, 2636, 292B,
2E2F, 3034, 323A, 3337, 353D, 3839, 3B3F, 3C3E.
\end{minipage}
\end{center}

\section{The nine Hamiltonian-pair certificates}\label{app:cycles}

Each row gives one-line cyclic notation for $\pi_{j,i}=m_j^{-1}m_i$.
Every row contains each even-parity vertex exactly once; the final arrow closes
the cycle at \texttt{00}.

\begin{center}
\begin{minipage}{0.98\textwidth}
\footnotesize\ttfamily\raggedright
$\pi_{4,1}$: 00 09 0A 2E 27 2B 1B 0F 3F 39 18 30 21 35 3C 3A
12 17 36 1E 1D 0C 06 14 11 33 22 03 05 24 2D 28 $\to$ 00.

\medskip
$\pi_{5,1}$: 00 11 17 27 03 0F 2D 39 1D 09 2B 33 12 36 3F 3C
18 14 35 24 28 0A 3A 1B 1E 0C 2E 06 05 21 30 22 $\to$ 00.

\medskip
$\pi_{6,1}$: 00 05 27 21 11 03 06 24 0C 0A 22 12 1E 14 1D 2D
2B 3F 35 30 28 09 0F 2E 36 3C 39 1B 17 33 3A 18 $\to$ 00.

\medskip
$\pi_{4,2}$: 00 21 09 28 3A 1E 06 03 0A 12 30 22 2B 2E 2D 39
1B 33 35 24 27 3F 36 3C 1D 0F 0C 14 17 05 11 18 $\to$ 00.

\medskip
$\pi_{5,2}$: 00 22 03 2B 3A 3F 27 2D 3C 0C 05 35 21 11 1D 1E
2E 28 18 0A 1B 17 0F 09 39 33 30 12 14 36 24 06 $\to$ 00.

\medskip
$\pi_{6,2}$: 00 28 39 3F 33 11 1B 03 0F 2D 35 27 2E 0C 24 36
30 3A 3C 14 1E 0A 18 09 2B 22 21 05 1D 17 06 12 $\to$ 00.

\medskip
$\pi_{4,3}$: 00 14 3C 2D 24 21 28 2E 1E 12 33 2B 3F 0F 06 17
11 30 3A 1B 0A 03 09 18 1D 39 35 36 22 27 05 0C $\to$ 00.

\medskip
$\pi_{5,3}$: 00 05 09 1D 3C 28 3A 33 03 11 14 24 22 06 36 12
17 35 27 0F 2E 3F 1E 1B 2B 2D 21 39 30 18 0C 0A $\to$ 00.

\medskip
$\pi_{6,3}$: 00 24 28 22 36 3A 3F 17 1D 35 33 21 2B 2E 3C 0C
09 1B 0F 0A 12 03 05 2D 27 06 1E 18 14 30 39 11 $\to$ 00.
\end{minipage}
\end{center}

\section{Certificates for the six nonedges}\label{app:nonedges}

For each within-part pair, the following table exhibits one proper alternating
cycle. Thus the corresponding union cannot be a Hamilton cycle.

\begin{center}
\begin{tabular}{cl}
\toprule
Pair & Proper cycle (hexadecimal vertices)\\
\midrule
$M_1,M_2$ & \texttt{24 2C 2E 26 $\to$ 24}\\
$M_1,M_3$ & \texttt{14 15 17 37 35 34 $\to$ 14}\\
$M_2,M_3$ & \texttt{1D 1F 3F 37 35 25 2D 3D $\to$ 1D}\\
$M_4,M_5$ & \texttt{00 08 0A 0B 2B 23 03 02 $\to$ 00}\\
$M_4,M_6$ & \texttt{0C 0D 2D 2C $\to$ 0C}\\
$M_5,M_6$ & \texttt{01 11 31 30 34 24 04 05 $\to$ 01}\\
\bottomrule
\end{tabular}
\end{center}

\bibliographystyle{plain}
\bibliography{references}

@article{behague2019,
  author  = {Behague, Natalie C.},
  title   = {Semi-perfect 1-factorizations of the hypercube},
  journal = {Discrete Mathematics},
  volume  = {342},
  number  = {6},
  pages   = {1696--1702},
  year    = {2019},
  doi     = {10.1016/j.disc.2019.01.035},
  note    = {\url{https://doi.org/10.1016/j.disc.2019.01.035}},
  eprint  = {1811.06389},
  archivePrefix = {arXiv},
  primaryClass  = {math.CO}
}

@misc{lambard2026q6archive,
  author       = {Lambard, Guillaume},
  title        = {Certificate and Verifiers for a 3-Semi-Perfect
                  1-Factorization of the Six-Dimensional Hypercube},
  year         = {2026},
  publisher    = {Zenodo},
  doi          = {10.5281/zenodo.21392425},
  howpublished = {\url{https://doi.org/10.5281/zenodo.21392425}},
  note         = {Source repository:
                  \url{https://github.com/GLambard/q6-semi-perfect-factorization}}
}

\end{document}